\documentclass{aptpub}

\usepackage{amsmath,mathtools,nicefrac,xcolor}
\newcommand{\ep}{{\epsilon}}

\numberwithin{equation}{section}

\authornames{C. Houdr\'e, G. Kerchev}
\shorttitle{Convergence of LCS in Markov models}

\begin{document}

\title{On the rate of convergence for the length of the longest common subsequences in hidden Markov models}

\authorone[Georgia Institute of Technology]{C. Houdr\'e}
\addressone{School of Mathematics,  686 Cherry St. NW, Atlanta, GA 30332-0160 USA}
\emailone{houdre@math.gatech.edu}

\authortwo[Georgia Institute of Technology]{G. Kerchev}
\emailtwo{kerchev@math.gatech.edu}

\begin{abstract} Let $(X, Y) = (X_n, Y_n)_{n \geq 1}$ be the output process generated by a hidden chain $Z = (Z_n)_{n \geq 1}$, where $Z$ is a finite state, aperiodic, time homogeneous, and irreducible Markov chain. Let $LC_n$ be the length of the longest common subsequences of $X_1, \ldots, X_n$ and $Y_1, \ldots, Y_n$. Under a  mixing hypothesis, a rate of convergence result is obtained for $\mathbb{E}[LC_n]/n$.
\end{abstract}

\keywords{ Longest Common Subsequences; Rate of Convergence; Hidden Markov Model; Hoeffding Inequality;  Mixing Conditions}

\ams{  60C05, 05A05}{60F10}

\section{Introduction}

Longest common subsequences are often a key measure of similarity between two strings of letters. For two finite sequences  $(X_1, \ldots, X_n)$ and $(Y_1, \ldots, Y_m)$ taking values in a finite alphabet $\mathcal{A}$, the object of study is $LCS(X_1, \ldots, X_n; Y_1, \ldots, Y_m)$, the length of the longest common subsequences of $X_1, \ldots, X_n$ and $Y_1, \ldots, Y_m$, which is  abbreviated as $LC_n$ when $n = m$. Clearly $LC_n$ is  the largest $k$ such that there exist  $1 \leq i_1 < \cdots <  i_k \leq n$ and $1 \leq j_1 < \cdots < j_k \leq n$ with 

\begin{equation}
\notag  X_{i_s} = Y_{j_s} \text {, for all } s = 1, 2, 3, \ldots, k. 
\end{equation}

\noindent  For two independent words sampled independently and uniformly at random from the alphabet, Chv\'atal and Sankoff~\cite{CS}  proved that $\lim_{n \to \infty} \mathbb{E}[LC_n]/n = \gamma^*$  and provided upper and lower bounds on $\gamma^*$. This was followed by Alexander~\cite{A} who obtained, for iid draws, the following generic rate of convergence result:

\begin{equation}\label{origin_KA}
 n \gamma^* - C \sqrt{n \log n}  \leq \mathbb{E}[LC_n] \leq n \gamma^*,
\end{equation}

\noindent where $C > 0$ is an absolute constant.

\noindent From a practical point of view the  independence assumptions, both between words and also among draws,  has to be relaxed as they are often lacking. One such instance is in the field of computational biology where one compares similarities between two biological sequences. In particular alignments of those sequences need to be qualified as occurring by chance or because of a structural relation. One way to generate alignments is with a hidden Markov model (HMM). The states of the hidden chain account for a match between two elements in $X$ and $Y$ or for an alignment of an element with a gap. Given $X$ and $Y$ one can find the most probable alignment using the Viterbi algorithm. This model is particularly useful when the similarity between $X$ and $Y$ is  weak. In this case standard methods for pairwise alignment often fail to identify the correct alignment or test for its significance. With a hidden Markov model  one can evaluate the total probability that $X$ and $Y$ are aligned by summing up over all alignments, and this sum can be efficiently computed with the Forward algorithm.  For more information we refer the reader to Chapter 4 in~\cite{DEKM}.

\noindent There are very few results on the asymptotics of the longest common subsequences in a model exhibiting dependence properties. A rare instance is due to Steele~\cite{S} who showed the convergence of $\mathbb{E}[LC_n]/n$ when $(X,Y)$  is a random sequence for which there is a stationary ergodic coupling, e.g., an irreducible, aperiodic, positive recurrent Markov chain. The present paper studies the longest common subsequences for strings exhibiting a different Markov relation, namely we study the case when  $(X,Y)$ is  emitted by a latent  Markov chain $Z$, i.e., when $(Z, (X,Y))$ is a hidden Markov model.  Note that this framework includes the special case when $(Z, X)$ and $(Z', Y)$ are hidden Markov models, with the same parameters, while $Z$ and $Z'$ are independent. In our setting, mean convergence is quickly proved in Section~\ref{s:CM}. Then, the main contribution is a rate of convergence result, obtained in Section~\ref{s:RC}, which   recovers, in particular,~\eqref{origin_KA}.

\noindent \emph{Throughout this manuscript our probability space $(\Omega, \mathcal{F}, \mathbb{P})$ is assumed to be rich enough to consider all the random variables we are studying. }

\section{Mean convergence}\label{s:CM} 
Recall that a hidden Markov model $(Z, V)$ consists of a Markov chain $Z = (Z_n)_{n \geq 1}$ which emits the observed variables $V = (V_n)_{n \geq 1}$. The possible states in $Z$ are each associated with a distribution on the values of $V$. In other words the observation $V$ is a mixture model where the choice of the mixture component for each observation depends on the component of the previous observation. The mixture components are given by the sequence $Z$.  Note also that given $Z$, $V$ is a Markov chain. For such a model our first easy result asserts the mean convergence of $LC_n$.

\begin{prop}\label{prop:mean}
Let $Z$ be an aperiodic, irreducible, time homogeneous finite state space Markov chain. Let $\mu$, $P$, and  $\pi$ be respectively the initial distribution, transition matrix and stationary distribution of $Z$. Let each $Z_n$, $n \geq 1$, generate a pair $(X_n, Y_n)$ according to a distribution associated to the state of $Z_n$, i.e., let $(Z, (X, Y))$ be a hidden Markov model, where $X = (X_n)_{n \geq 1}$ and $Y = (Y_n)_{n \geq 1}$. Further,  for all $i \geq 1$ and $j \geq 1$, let   $X_i$ and $Y_j$ take their values in the common finite alphabet $\mathcal{A}$ and let there exists $a \in \mathcal{A}$, such that $\mathbb{P} (X_i = Y_j = a) > 0$, for some $i\geq 1$ and $j \geq 1$. Then,  
\begin{equation}
\notag \lim_{n \to \infty} \frac{\mathbb{E}[LC_n]}{n} = \gamma^*, 
\end{equation}

\noindent where $\gamma^* \in (0, 1]$. 

\end{prop}

\begin{proof} If $\mu = \pi$, the sequence $(X, Y)$ is stationary and therefore by superadditivity and  Fekete's lemma or Kingman's subadditivity theorem (see~\cite{St}) imply: 
\begin{equation}\label{eq:mean}
 \lim_{n \to \infty} \frac{\mathbb{E}[LC_n]}{n} = \sup_{k \geq 1} \frac{\mathbb{E}[LC_k]}{k} =  \gamma^*, 
\end{equation}

\noindent for some $\gamma^* \in (0, 1]$.  When $\mu \neq \pi$, a coupling technique will prove the result. Let $\overline{Z}$ be a Markov chain with initial and  stationary distribution $\pi$ and having the same transition matrix $P$ as the chain $Z$. Assume, further, that the emission probabilities are the same for $Z$ and $\overline{Z}$ and denote by $(\overline{Z}, (\overline{X}, \overline{Y}))$ the corresponding HMM. Next consider the coupling $(Z, \overline{Z})$ where the two chains stay together after the first time $i$ for which $Z_i = \overline{Z}_i$, and let $\tau$ be the meeting time of $Z$ and $\overline{Z}$. Next, and throughout,  let $X^{(n)} \coloneqq (X_1, \ldots, X_n)$ and similarly for $Y^{(n)}, \overline{X}^{(n)}$ and $\overline{Y}^{(n)}$. Since $LCS(X^{(n)}; Y^{(n)}) - LCS(\overline{X}^{(n)}; \overline{Y}^{(n)}) \leq n$, then for any $K>0$, 
\begin{align}
\notag    |& \mathbb{E}  [ LCS(X^{(n)}; Y^{(n)}) - LCS(\overline{X}^{(n)}; \overline{Y}^{(n)})] | \\
\notag = &   \quad \bigg| \mathbb{E} \left[  [LCS(X^{(n)}; Y^{(n)}) - LCS(\overline{X}^{(n)}; \overline{Y}^{(n)})] \mathbf{1}_{\tau > K} \right] \\
\notag & \quad \quad \quad  \quad \quad   + \mathbb{E} \left[ [  LCS(X^{(n)}; Y^{(n)}) - LCS(\overline{X}^{(n)}; \overline{Y}^{(n)})] \mathbf{1}_{\tau \leq K} \right] \bigg| \\
\notag  \leq &  \quad  n \mathbb{P}(\tau > K) + K + \left| \mathbb{E}\left[ [ LCS^K(X^{(n)}; Y^{(n)}) - LCS^K(\overline{X}^{(n)}; \overline{Y}^{(n)}) ] \mathbf{1}_{\tau \leq K} \right]\right| \\
\label{eq:mean_coupling}  \leq &  \quad n \mathbb{P}(\tau > K) + K,
\end{align}

\noindent where $LCS^K(\cdot;\cdot)$ is now the length of the longest common subsequences restricted to the letters $X_i$ and $Y_i$, for $i > K$,  noting also that when $\tau \leq K$, then $LCS^K(X^{(n)}; Y^{(n)})$ and $LCS^K(\overline{X}^{(n)}; \overline{Y}^{(n)})$ are identically distributed. If $K \in (mk, m(k+1)]$, for some $m \geq 0$, by an argument going back to Doeblin~\cite{D} (see also ~\cite{Thorisson}),
\begin{align}
\notag \mathbb{P}& (\tau > K) \\
\notag & \leq \quad  \mathbb{P} (Z_k \neq \overline{Z_k}, Z_{2k} \neq \overline{Z_{2k}}, \ldots, Z_{mk} \neq \overline{Z_{mk}}) \\
\notag & = \quad \mathbb{P} (Z_k \neq \overline{Z_k}) \mathbb{P}( Z_{2k} \neq \overline{Z_{2k}} | Z_k \neq \overline{Z_k})  \cdots \mathbb{P}( Z_{mk} \neq \overline{Z_{mk}}| Z_{(m-1)k} \neq \overline{Z_{(m-1)k}}) \\
\notag & \leq  \quad (1 - \ep)^{m-1} \\ 
\label{eq:doeb} & \leq  \quad c \alpha^K,
\end{align}

\noindent where $\alpha = \sqrt[k]{1 - \ep} \in (0,1)$ and $c = 1/ (1 - \ep)^2$. Therefore,  $\tau$ is finite with probability one. Choosing $K = \sqrt{n}$, yields $\mathbb{P}(\tau > K) + K/n \to 0$ and finally  $\mathbb{E}[LC_n]/n \to \gamma^*$, as $n \to \infty$.

\noindent  Clearly, $\mathbb{E}[LC_n] \leq n$ and to see that $\gamma^* > 0$, note first that, by aperiodicity and irreducibility, $P^k \geq  \ep$, for some fixed $k$ and $\ep > 0$, i.e., all the entries of the matrix $P^k$ are larger than some positive quantity $\ep$. Therefore $\mathbb{P}(X_1 = Y_{k + 1}) > p$, for some $p  = p(k, \ep) > 0$. Now,
\begin{equation}\label{eq:gamma_lb}
 LC_{nk+1} \geq \mathbf{1}_{X_1 = Y_{k+1}} + \mathbf{1}_{X_{k+1} = Y_{2k+1}} + \cdots + \mathbf{1}_{X_{(n-1)k + 1} = Y_{nk + 1}},
\end{equation}

\noindent hence 
\begin{equation}
\notag  \frac{n p}{nk + 1} \leq \frac{\mathbb{E}[LC_{nk+1} ]}{nk+1}.    
\end{equation}

\noindent Letting $n \to \infty$ implies that $\gamma^* \in [p/(k+1), 1] \subset (0,1]$, since $p > 0$. 

\end{proof}

\begin{rem}\label{rem:first}

\noindent (i) Under a further assumption, one can show that $\gamma^* > \mathbb{P}(X_1 = Y_1)$. Indeed, assume that for all $x, y \in \mathcal{A}, z \in \mathcal{S}$,  $\mathbb{P} (X_i = x, Y_i  = y | Z_i =z) = \mathbb{P} (X_i = y , Y_i = x | Z_i  = z) > 0$, and let $Z$ be started at the stationary distribution. Then for any $n \geq 2$, 
\begin{align}
\notag \mathbb{E}[LC_n]  \geq  & \quad \mathbb{E}[LC_{n-2} \mathbf{1}_{X_n = Y_n, X_{n-1} = Y_{n-1}}] + 2 \mathbb{P}( X_n = Y_n, X_{n-1} = Y_{n-1}) \\
\notag & + \mathbb{E}[LC_{n-2} \mathbf{1}_{X_n = Y_n, X_{n-1} \neq Y_{n-1}}] +  \mathbb{P}( X_n \neq Y_n, X_{n-1} = Y_{n-1})  \\
\notag & + \mathbb{E}[LC_{n-2} \mathbf{1}_{X_n \neq Y_n, X_{n-1} = Y_{n-1}} ]+  \mathbb{P}( X_n = Y_n, X_{n-1} \neq Y_{n-1})  \\
\notag & + \mathbb{E}[LC_{n-2} \mathbf{1}_{X_n \neq Y_n, X_{n-1} \neq Y_{n-1}} ]+  \mathbb{P}( X_n \neq Y_n, X_{n-1} \neq Y_{n-1}, X_n = Y_{n-1})  \\
\notag > & \quad \mathbb{E}[LC_{n-2}] +  \mathbb{P}( X_n = Y_n) + \mathbb{P}(X_{n-1} = Y_{n-1})\\
\notag = & \quad \mathbb{E}[LC_{n-2}] + 2 \mathbb{P} (X_1 = Y_1),
\end{align}

\noindent by stationarity. Therefore, iterating, still using stationarity, and since $\mathbb{E}[LC_0] = 0$ while $\mathbb{E}[LC_1] = \mathbb{P}(X_1 = Y_1)$, it follows that for  $n \geq 2$, $\mathbb{E}[LC_n] > n \mathbb{P}(X_1 = Y_1)$. Finally, 
\begin{align}
\notag \gamma^* > \mathbb{P}(X_1 = Y_1) = \sum_{\alpha \in \mathcal{A} } \mathbb{P}(X_1 = \alpha) \mathbb{P}(Y_1 = \alpha),
\end{align}

\noindent and this inequality is strict since Fekete's lemma, e.g., see~\cite{St}, ensures that $\gamma^* = \sup_{n} \mathbb{E}[LC_n]/n$.

\noindent (ii) Steele's general result, see~\cite{S}, asserts that Proposition~\ref{prop:mean} holds if there is a stationary ergodic coupling for  $(X, Y)$.  Such an example is when the sequences $X$ and $Y$ are generated by two  independent aperiodic, homogeneous and irreducible hidden Markov chains with the same parameters (and so the same emission probabilities). Indeed, at first, when the hidden chains $Z_X$ and $Z_Y$  generating respectively $X$ and $Y$ are started at the stationary distribution, convergence of $\mathbb{E}[LC_n]/n$ towards $\gamma^*$, follows from super-additivity and  Fekete's lemma (see~\cite{St}). As  previously,   $\gamma^* > 0$,  since the properties of the hidden chains imply~\eqref{eq:gamma_lb}. Then, when the initial distribution is not the stationary distribution, one can proceed with  arguments as above. In particular let $\tau_1$ and $\tau_2$ be the respective meeting times of the  chains $(Z_X, \overline{Z_X})$ and $(Z_Y, \overline{Z_Y})$, and let $\tau = \max (\tau_1, \tau_2)$. Then, equation~\eqref{eq:mean_coupling} continues to hold:
\begin{align}
\nonumber \left| \mathbb{E} [ LCS(X; Y) - LCS(\overline{X}; \overline{Y})] \right| & \leq n \mathbb{P}(\tau > K) + K \\
 \label{eq:error_2HMM}& \leq 2n \mathbb{P}(\tau_1 > K) + K.
\end{align}

\noindent Taking $K = \sqrt{n}$ and noting the exponential decay of $\mathbb{P}(\tau_1 > K)$ finishes the corresponding proof. 

\end{rem}

\section{Rate of convergence}\label{s:RC}

\noindent The previous section gives a  mean convergence result, we now deal with its rate. Again let $(X,Y)$ be the outcome of a hidden Markov chain $Z$ with $\mu, P$ and $\pi$ as initial distribution, transition matrix and stationary distribution respectively. In this section we impose the additional restriction that the emission distributions for all states in the hidden chain are symmetric (this is discussed further in Proposition~\ref{prop_asymetry} and in the Appendix), namely for all $x, y \in \mathcal{A}$ and all $z \in \mathcal{S}$, $\mathbb{P}(X_i = x, Y_i = y | Z_i = z) =\mathbb{P}(X_i = y, Y_i = x | Z_i = z)$. Symmetry clearly implies that the conditional law of $X$ given $Z$ and of $Y$ given $Z$ are the same since for all $x, y$ and $z$, 
\begin{align}
\notag \mathbb{P}(X_i=x|Z_i=z) = & \sum_{y \in \mathcal{A}} \mathbb{P}(X_i=x, Y_i = y|Z_i =z) = \sum_{y \in \mathcal{A}} \mathbb{P}(X_i=y, Y_i = x|Z_i =z) \\
\notag = & \quad \mathbb{P}(Y_i = x|Z_i =z).
\end{align}
\noindent  In turn this implies that $X_i$ and $Y_i$ are identically distributed.

  Moreover,  one needs to control the dependency between $X$ and $Y$ and a  way to do so is via the $\beta-$mixing coefficient, as given in Definition 3.3 of~\cite{Br} which we now recall.

\begin{defn}
Let  $\mathcal{F}_1$ and $\mathcal{F}_2$ be two $\sigma-$fields $\subset \mathcal{F}$, then the $\beta-$mixing coefficient, associated with these sub-$\sigma$-fields of $\mathcal{F}$, is given by:
\begin{equation}
\notag \beta(\mathcal{F}_1, \mathcal{F}_2) \coloneqq \frac{1}{2} \sup  \sum_{i =1}^I \sum_{j =1}^J | \mathbb{P} (A_i \cap B_j ) - \mathbb{P}(A_i) \mathbb{P}(B_j)|, 
\end{equation}
where the supremum is taken over all pairs of finite partitions $\{A_1, \ldots, A_I\}$ and $\{B_1, \ldots, B_J\}$ of $\Omega$ such that $A_i \in \mathcal{F}_1$, for all $i \in \{ 1, \ldots, I\}$, $I \geq 1$ and $B_j \in \mathcal{F}_2$ for all $j \in \{1, \ldots, J\}$, $J \geq 1$.
\end{defn}

\noindent In our case the above notion of $\beta-$mixing coefficient is adopted for the $\sigma-$fields generated by {\it two} sequences. Moreover, by~\cite[Proposition 3.21]{Br}, for a fixed $n \geq 1$, and since $X^{(n)} = (X_1, \ldots, X_n)$ and $Y^{(n)} = (Y_1, \ldots, Y_n)$ are discrete random vectors,
\begin{align}
 \notag \beta(n) \coloneqq & \quad  \beta\left(\sigma\left(X^{(n)}\right), \sigma\left(Y^{(n)}\right)\right) \\
\label{eq:beta_m} = & \quad  \frac{1}{2} \sum_{u \in \mathcal{A}^n} \sum_{v \in \mathcal{A}^n}  \left| \mathbb{P} \left(X^{(n)} = u, Y^{(n)} = v \right) - \mathbb{P}\left(X^{(n)} = u\right) \mathbb{P}\left(Y^{(n)} = v\right)\right|, 
\end{align}

\noindent  where $\sigma\left(X^{(n)}\right)$ and $\sigma\left(Y^{(n)}\right)$ are the $\sigma-$fields generated by $X^{(n)}$ and $Y^{(n)}$. Clearly $X^{(n)}$ and $Y^{(n)}$ are independent if and only if $\beta(n) = 0$. Further, set $\beta^* \coloneqq \lim_{n \to \infty} \beta(n)$, where the limit exists since $\beta(n)$ is non-decreasing, in $n$,  and $\beta(n) \in [0,1]$ (see Section 5 in~\cite{Br}).

\begin{rem}(i) Another definition of $\beta-$mixing coefficient  based on ``past" and ``future" is often studied in the literature, see, for instance,~\cite[Section 2]{B_s}. For  a single sequence of random variables $S = (S_k)_{k \in \mathbb{Z}}$  and for $- \infty \leq J \leq L \leq \infty$, let
\begin{align}
\nonumber \mathcal{F}_J^L \coloneqq \sigma(S_k , J \leq k \leq L ),
\end{align}

\noindent and for each $n \geq 1$, let 
\begin{align}
\nonumber \beta_n \coloneqq \sup_{j \in \mathbb{Z} } \beta( \mathcal{F}_{- \infty}^j, \mathcal{F}_{j + n}^{\infty}).
\end{align}

\noindent In particular~\cite[Theorem 3.2]{B_s} implies that if $S$ is a strictly stationary, finite-state Markov chain that is also irreducible and aperiodic, $\beta_n \to 0$ as $n \to \infty$. The mixing definition relevant to our approach is different and this limiting behavior does not follow. A further discussion of the  values of $\beta(n)$ is included in Remark~\ref{rem:final} (i).

\noindent (ii) One might also be interested to use the $\alpha-$mixing coefficient defined for $\sigma-$ fields $\mathcal{S}$ and $\mathcal{T}$ as:
\begin{align}
\nonumber \alpha(\mathcal{S} , \mathcal{T}) = 2 \sup \{ | Cov( \mathbf{1}_S, \mathbf{1}_T ) |: (S, T) \in \mathcal{S} \times \mathcal{T} \}
\end{align}

\noindent Suppose further that $\mathcal{T}$ has exactly $N$ atoms. The following holds (see~\cite{B_s} and~\cite[Theorem 1]{B_a}):
\begin{align}
\nonumber 2 \alpha(\mathcal{S}, \mathcal{T}) \leq  \beta(\mathcal{S}, \mathcal{T}) \leq (8N)^{1/2} \alpha(\mathcal{S}, \mathcal{T}).
\end{align}

\noindent However, for our setting the number of atoms $N$ will be $|\mathcal{A}|^n$, and since $\alpha(n) \coloneqq \alpha(\sigma(X^{(n)}) , \sigma(Y^{(n)}) )$ is increasing, a bound on $\beta(n)$ using the inequality above is useless.

\end{rem}

The following rate of convergence is our main result:


\begin{thm}\label{rate_main} Let $(Z, (X, Y))$ be a hidden Markov model, where the sequence $Z$ is an aperiodic time homogeneous and irreducible Markov chain with  finite state space $\mathcal{S}$. Let the distribution of the pairs $(X_i, Y_i)$,  $i = 1, 2, 3, \ldots$, be symmetric for all states in $Z$. Then, for all $n \geq 2$,
\begin{equation}\label{rate_bound_l}
 \frac{\mathbb{E}[LC_n]}{n} \geq \gamma^* - 2\beta^* -  C \sqrt{\frac{\ln n}{n} } - \frac{2}{n}  - (1 - \mathbf{1}_{\mu = \pi}) \left( \frac{1}{\sqrt{n}}+  c\alpha^{\sqrt{n}}   \right) , 
\end{equation}

\noindent where $\alpha \in (0,1), c > 0$ are constants as in~\eqref{eq:doeb} and $C > 0$. All constants  depend on the parameters of the model but not on $n$. Moreover with the same $\alpha$ and $c$, 
\begin{align}
 \label{rate_bound_u}\frac{\mathbb{E}[LC_n]}{n} \leq \gamma^*  + (1 - \mathbf{1}_{\mu = \pi}) \left( \frac{1}{\sqrt{n}}+ c \alpha^{\sqrt{n}}  \right).
\end{align}

\end{thm}

\medskip

\noindent A key ingredient in proving Theorem~\ref{rate_main} is a Hoeffding-type inequality for Markov chains, a particular case of a result due to Paulin~\cite{DP}, which is now recalled.  It relies on the  mixing time $\tau(\ep)$ of the Markov chain $Z$ given by
\begin{equation}
\notag \tau(\ep) \coloneqq \min \{t \in \mathbb{N} : \overline{d}_Z(t) \leq \ep \},
\end{equation} 

\noindent where
\begin{equation}
\notag \overline{d}_Z(t) \coloneqq \max_{1 \leq i \leq N - t} \sup_{x, y \in \Lambda_i} d_{TV} ( \mathcal{L}  ( Z_{i+t}| Z_i = x ) ,  \mathcal{L} ( Z_{i+t} | Z_i = y)),
\end{equation}

\noindent and where $d_{TV}(\mu, \nu) = \frac{1}{2}\sum_{x \in \Omega} | \mu(x) - \nu(x) |$ is the total variation distance between the two probability measures $\mu$ and $\nu$ on the finite set $\Omega$.

\begin{lem}\label{lem:hoeff} Let $M \coloneqq (M_1, \ldots , M_N)$ be a (not necessarily time homogeneous) Markov chain, taking values in a Polish space $\Lambda = \Lambda_1 \times \cdots \times \Lambda_N$, with mixing time $\tau(\ep)$,  $0 \leq \ep \leq 1$. Let 
\begin{equation} 
\notag \tau_{min} \coloneqq \inf_{0 \leq \ep < 1} \tau(\ep)  \left(\frac{2 - \ep}{1 - \ep}\right)^2, 
\end{equation}

\noindent and let  $f : \Lambda \to \mathbb{R}$ be such that there is $c \in \mathbb{R}_+^N$ with  $|f(u) - f(v)| \leq \sum_{i = 1}^N c_i \mathbf{1}_{u_i \neq v_i}$  . Then for any $t \geq 0$,
\begin{equation}\label{eq:conc}
\mathbb{P}(f(M) - \mathbb{E} f(M) \geq t) \leq  \exp \left(\frac{-2t^2}{ \tau_{min} \sum_{i = 1}^N c_i^2}\right). 
\end{equation}
 
\end{lem}

\noindent For our purposes,  the Hoeffding-type inequality used below follows directly from~\eqref{eq:conc}  once one notes that $(Z_i, X_i, Y_i)_{i \geq 1}$ is jointly a Markov chain on a bigger state space. Let $\tau(\ep)$ be the mixing time of this chain. Taking $f$ to be the length of the longest common subsequences of $X_1, \ldots, X_n$ and $Y_1, \ldots, Y_n$ we have $c = ((0, \ldots, 0), (1, \ldots, 1)) \in \mathbb{R}^{n} \times \mathbb{R}^{2n}$, since $f$ is a function of $Z$, $X$ and $Y$, whose values do not depend  on $Z$.  Letting $A \coloneqq \sqrt{\tau_{min}/ 2 }$, ~\eqref{eq:conc} becomes,

\begin{equation}\label{hoeff}
 \mathbb{P}( LC_n - \mathbb{E}[LC_n]  \geq t) \leq  \exp \left[ \frac{-t^2 }{A^2 n} \right],
\end{equation}

\noindent for all $t \geq 0$.





\begin{rem}\label{Hoeff_HMM}(i) When $X$ and $Y$ are generated by two independent hidden chains $Z^X$ and $Z^Y$, the same reasoning yields~\eqref{hoeff} where now $\tilde{\tau}(\ep)$ is the mixing time of the chain $(Z^X_n, Z^Y_n, X_n, Y_n)_{n \geq 1}$.  

\noindent (ii)  The mixing time $\tau(\ep)$ of $(Z_n, X_n, Y_n)_{n \geq 1}$ is the same as  the mixing time $\tilde{\tau}(\ep)$ of the chain $(Z_n)_{n \geq 1}$. Two proofs of this fact are provided in the Appendix. 
\end{rem}

\begin{proof}[Proof of Theorem \ref{rate_main}]\label{rate_1}

  First recall  a  result  of Berbee~\cite{B}, see also~\cite[Theorem 1, Section 1.2.1]{PD},~\cite[Chapter 5]{ER}, and~\cite{G}, asserting that on our probability space, which is rich enough, there exists $Y^{*(n)} \coloneqq (Y_1^*, \ldots, Y_n^*)$, independent of $(Z , X)^{(n)} = ((Z_1, X_1), \ldots , (Z_n, X_n))$, having the same law as $Y^{(n)} = (Y_1, \ldots, Y_n)$ and such that 
\begin{equation}\label{thm:beta}
 \mathbb{P}( Y ^{(n)}\neq Y^{*(n)}) = \beta(n),
\end{equation}

\noindent  where $\beta(n) = \beta(\sigma((Z, X)^{(n)}), \sigma(Y^{(n)}))$ is the $\beta-$mixing coefficient of $(Z, X)^{(n)}$ and $Y^{(n)}$.  Note also that if  $ (Y_i)_{i \geq 1}$ is stationary, then $(Y_1^*, \ldots, Y_k^*)$ and $(Y_{\ell}^*, \ldots, Y_{\ell+k-1}^*)$ are identically distributed, for every $\ell, k \geq 1$,   and that  if $(X^{(n)}, Y^{(n)})$ is symmetric, then so is $(X^{(n)}, Y^{*(n)})$ where $X^{(n)} = (X_1, \ldots, X_n)$. Note finally that this implies that $Y^{*(n)}$ is independent of both $X^{(n)}$ and $Z^{(n)} = (Z_1, \ldots, Z_n)$.

\noindent Next, fix $k \in \mathbb{N}$, the idea of the proof is to relate $\mathbb{E}[LC_{kn}]$ to $\mathbb{E}[LC_{2n}]$. For $k = 4$, this is done in the i.i.d case in ~\cite{R}. However, we wish to take $k \to \infty$ and therefore follow  arguments presented for the i.i.d case  in~\cite{LMT}. Call $(\nu, \tau) \coloneqq (\nu_1, \ldots, \nu_r, \tau_1, \ldots \tau_r)$  an $r-$ partition with $k \leq r \leq \lceil 2kn/(2n -1) \rceil$  if 
\begin{equation}\label{def:partition}
\begin{split}
&1  = \nu_1 \leq \nu_2 \leq \cdots \leq \nu_{r + 1} = kn + 1, \\
&1 = \tau_1 \leq \tau_2 \leq \cdots \leq \tau_{r + 1} = kn + 1,\\ 
&(\nu_{j+1} - \nu_j) + (\tau_{j+1} - \tau_j) \in \{(2n-1, 2n\}, \text{ for  } j \in [1, r-1], \\
&(\nu_{r+1} - \nu_{r }) + (\tau_{r+1} - \tau_{r} ) < 2n. 
\end{split}
\end{equation}

\noindent Let $\mathcal{B}_{k,n}^r$ be the set of all $r-$ partitions defined as above and let
\begin{equation}
\notag \mathcal{B}_{k, n} = \bigcup_{r = k}^{\lceil 2kn / (2n-1) \rceil} \mathcal{B}_{k, n}^r.
\end{equation}

\noindent If $(\nu, \tau)$ is an $r-$partition, setting 
\begin{equation}
\notag LC_{kn} (\nu, \tau) \coloneqq \sum_{i = 1}^r LCS (X_{\nu_i} , \ldots, X_{\nu_{i+1} - 1}; Y_{\tau_i} , \ldots, Y_{\tau_{i+1} - 1}),
\end{equation}
\noindent then:
\begin{equation}
\notag LC_{kn} = \max_{(\nu, \tau) \in \mathcal{B}(k, n)} LC_{kn}(\nu, \tau).
\end{equation}

\noindent Let $\nu_{i+1} - \nu_i = n-m$, $\tau_{i+1} - \tau_i \leq n+m$ for $m \in (-n, n)$ and $\tau_i - \nu_i = \ell$. Then,
\begin{align}
\notag & \mathbb{E}[LCS (X_{\nu_i} , \ldots, X_{\nu_{i+1} - 1}; Y_{\tau_i} , \ldots, Y_{\tau_{i+1} - 1})] \\
\label{eq:1}& =  \quad \mathbb{E}[LCS(X_1, \ldots , X_{n-m}; Y_{\ell} , \ldots, Y_{\ell+n+m-1})] \\
\notag & \leq \quad \mathbb{E}\left[LCS(X_1, \ldots , X_{n-m}; Y_{\ell}^* , \ldots, Y_{\ell+n+m-1}^*) \mathbf{1}_{Y^{(kn)} = Y^{*(kn)}}\right]  \\
\label{eq:1a}  & \quad \quad \quad \quad \quad \quad \quad  + \min (n-m, n+m) \mathbb{P}\left(Y^{(kn)} \neq Y^{*(kn)}\right	) \\
\label{eq:2} & \leq   \quad \mathbb{E}[LCS(X_1, \ldots , X_{n-m}; Y_{\ell}^* , \ldots, Y_{\ell+n+m-1}^*)] + n \beta(kn).  
\end{align}

\noindent  In the last expression the LCS is now a function of two independent sequences. Stationarity implies~\eqref{eq:1} and $LCS(X_1, \ldots, X_{n-m};  Y_{\ell}^* , \ldots, Y_{\ell+n+m-1}^*) \leq \min (n-m, n+m)$ entails~\eqref{eq:1a}. The error term $n \beta(kn)$ in~\eqref{eq:2} follows from an application of Berbee's result~\eqref{thm:beta}. The same properties also imply 
\begin{align}
\notag & \mathbb{E}[LCS(X_1, \ldots , X_{n-m}; Y_{\ell}^* , \ldots, Y_{\ell+n+m-1}^*)] \\
\notag & = \quad \mathbb{E}[LCS(X_1, \ldots , X_{n-m}; Y_1^* , \ldots, Y_{n+m}^*)] \\
\label{eq:3} & \leq \quad \mathbb{E}[LCS(X_1, \ldots , X_{n-m}; Y_1 , \ldots, Y_{n+m})] + n\beta(kn),
\end{align}

\noindent and
\begin{align}
\notag & \mathbb{E}[LCS(X_1, \ldots , X_{n-m}; Y_{\ell}^* , \ldots, Y_{\ell+n+m-1}^*)] \\
\label{eq:4} & = \quad \mathbb{E}[LCS(X_1, \ldots , X_{n+m}; Y_1^* , \ldots, Y_{n-m}^*)] \\
\label{eq:5} & \leq \quad \mathbb{E}[LCS(X_{n-m+1}, \ldots , X_{2n}; Y_{n+m+1} , \ldots, Y_{2n})] +n \beta(kn),
\end{align}

\noindent  where the symmetry of the distributions of $X$ and $Y^*$ is used to get~\eqref{eq:4}. Next by superadditivity of the LCSs as well as~\eqref{eq:2},~\eqref{eq:3} and~\eqref{eq:5}, 
\begin{align}
\notag \mathbb{E} & [ LCS (X_{\nu_i} , \ldots, X_{\nu_{i+1} - 1}; Y_{\tau_i} , \ldots, Y_{\tau_{i+1} - 1})]  \\
\notag  \leq & \quad \frac{1}{2} \bigg( \mathbb{E}[LCS(X_1, \ldots , X_{n-m}; Y_1 , \ldots, Y_{n+m})]  \\
\notag & \quad \quad + \mathbb{E}[LCS(X_{n-m+1}, \ldots , X_{2n}; Y_{n+m+1} , \ldots, Y_{2n})] + 2 n \beta(kn) \bigg) + n \beta(kn) \\
\notag  \leq & \quad  \frac{1}{2} \bigg( \mathbb{E}[LC_{2n}] + 2 n \beta(kn) \bigg) + n \beta(kn)\\ 
\label{eq:exchange} = & \quad \frac{1}{2} \mathbb{E}[LC_{2n}] + 2 n \beta(kn).
\end{align}

\noindent This inequality is key  to the proof, since it yields an upper bound on $\mathbb{E}[LC_{kn}(\nu, \tau)]$ in terms of $\mathbb{E}[LC_{2n}]$, a quantity that does not depend on the partitioning $(\nu, \tau)$. A similar result is central to the proof of the rate of convergence in the independent setting~\cite{A}. However, independence allows one to get~\eqref{eq:exchange} directly without the mere presence of or the need to introduce $\beta$-mixing coefficients. Moreover, our approach is more direct. Applying  Hoeffding's inequality and summing over all partitions provide a relation between $\mathbb{E}[LC_{kn}]$ and $\mathbb{E}[LC_{2n}]$ which can be used to get the rate of convergence. Indeed,
\begin{equation}
\notag \mathbb{E}[LC_{kn} (\nu, \tau)] \leq \frac{r}{2} (\mathbb{E}[LC_{2n}] + 4n \beta(kn) ) \leq \frac{1}{2}\left\lceil \frac{2kn}{ 2n-1 }\right\rceil (\mathbb{E}[LC_{2n}] + 4n \beta(kn) ).
\end{equation}

\noindent In addition, for $t > 0$,
\begin{align}
\notag \mathbb{P}& \left(LC_{kn}(\nu, \tau) - \frac{1}{2} \left\lceil \frac{2kn}{ 2n-1 }\right\rceil \left( \mathbb{E}[LC_{2n}] + 4n \beta(kn)\right) > t kn \right) \\
\notag & \leq  \quad \mathbb{P}\left(LC_{kn}(\nu, \tau) - \mathbb{E}[LC_{kn} (\nu, \tau)]  > t kn \right) \\
 & \leq  \quad  \exp \left[ - \frac{t^2 k n } {A^2} \right], 
\end{align}

\noindent where the second inequality follows from Lemma~\ref{lem:hoeff}.  Next note that:
\begin{align}
\notag \mathbb{P} & \left(LC_{kn}  - \frac{1}{2} \left\lceil \frac{2kn}{ 2n-1 }\right\rceil \left( \mathbb{E}[LC_{2n}] + 4n \beta(kn)\right) > t kn \right) \\ 
\notag & =  \sum_{(\nu, \tau) \in \mathcal{B}_{k, n} } \mathbb{P}\left(LC_{kn}(\nu, \tau) - \frac{1}{2} \left\lceil \frac{2kn}{ 2n-1 }\right\rceil \left( \mathbb{E}[LC_{2n}] + 4n \beta(kn)\right) > t kn \right)  \\ 
\notag & \leq  |\mathcal{B}_{k, n} |  \exp \left[ - \frac{t^2 k n } {A^2} \right].
\end{align}

\noindent The above can be rewritten as: 
\begin{equation}
\notag \mathbb{P}\left(\frac{LC_{kn} }{kn} > t + \frac{1}{k} \left\lceil \frac{2kn}{ 2n-1 }\right\rceil \left( \frac{\mathbb{E}[LC_{2n}]}{2n} + 2 \beta(kn)\right)  \right) \leq    |\mathcal{B}_{k, n} | \exp \left[ - \frac{t^2 k n } {A^2} \right].
\end{equation}

\noindent  Then, since $LC_{kn}\leq kn$,
\begin{align}
\notag \mathbb{E} \left[\frac{LC_{kn} }{kn}\right] \leq  &\quad  t + \frac{1}{k} \left\lceil \frac{2kn}{ 2n-1 }\right\rceil \left(\frac{\mathbb{E}[LC_{2n}]}{2n} + 2 \beta(kn) \right)  \\
\notag & \quad \quad \quad \quad + \mathbb{P}\left(\frac{LC_{kn} }{kn} > t + \frac{1}{k} \left\lceil \frac{2kn}{ 2n-1 }\right\rceil \frac{\mathbb{E}[LC_{2n}]}{2n}   \right) \\
\label{eq:hoeff} \leq & \quad t + \frac{1}{k} \left\lceil \frac{2kn}{ 2n-1 }\right\rceil \left( \frac{\mathbb{E}[LC_{2n}]}{2n} + 2 \beta(kn)\right)  + |\mathcal{B}_{k, n} | \exp \left[ - \frac{t^2 k n } {A^2} \right].
\end{align}

\noindent Next a bound on $|\mathcal{B}_{k, n}|$ is obtained using  methods as in~\cite{LMT}. Recall that $k \leq r \leq \lceil 2kn/(2n-1) \rceil$ and that $\mathcal{B}_{k, n} = \bigcup_{r = k}^{\nicefrac{2kn}{2n-1}} \mathcal{B}_{k, n}^r$. Now 
\begin{equation}\label{B_knr:bound}
 | \mathcal{B}_{k, n}^r|   \leq 2^{r-1} 2n \binom{nk + r -1}{r- 1}.
\end{equation}
  
\noindent Indeed, the sum of sizes of the partition on the $X$ side should sum to $nk$ which gives a factor of less than $\binom{nk + r - 1}{ r-1}$. Also for each choice of the first $r-1$ elements of the partition on the $X$ side we have at most $2$ choices on the $Y$ side. The last interval can take at most $2n$ values, as per ~\eqref{def:partition}. Recall Stirling's formula (see~\cite{F}), for $n \geq 1$, 
\begin{equation}
\notag n^n e^{-n} \sqrt{2 \pi n } e^{\nicefrac{1}{(12n+1)}} \leq n! \leq n^n e^{-n} \sqrt{2 \pi n } e^{\nicefrac{1}{12n}}. 
\end{equation}

\noindent Since in the end of the proof $k \to \infty$, this bound can be used in ~\eqref{B_knr:bound} to obtain:
\begin{equation}
 \begin{split}
\notag  |\mathcal{B}_{k,n}^r| & \leq (2^{r-1} 2n) \frac{(nk + r -1)^{nk + r -1 } \sqrt{2\pi (nk + r - 1)} e^{\nicefrac{1}{12(nk + r - 1)}}}{(r-1)^(r-1) \sqrt{2\pi (r-1)} e^{\nicefrac{1}{12(r-1)+1}} (nk)^{nk} \sqrt{2 \pi nk} e^{\nicefrac{1}{12(nk) + 1}}} \\
\notag & \leq 2^r n \frac{(nk+r-1)^{nk + r -1}}{(r-1)^{r-1} (nk)^{nk}}\\
\notag & \leq 2^r n \left(1 + \frac{nk}{r-1}\right)^{r-1} \left(1 + \frac{2}{2n-1}\right)^{nk} \\
\notag & \leq 2^r n \left(1 + n + \frac{n}{k-1}\right)^{\frac{2nk}{2n-1}} \left(\frac{2n+1}{2n-1}\right)^{nk}.
 \end{split}
\end{equation}

\noindent The last inequality in the above  expression  holds true since $k \leq r \leq \lceil 2kn/(2n-1) \rceil$. Then for $|\mathcal{B}_{k,n}|$ one gets:
\begin{equation}
\begin{split}
\notag |\mathcal{B}_{k, n}| & \leq \left(\frac{2nk}{2n-1} - k + 2\right) \max_r |\mathcal{B}_{k,n}^r| \\
\notag & \leq \left( \frac{k}{2n-1} + 2 \right) 2^r n \left(1 + n + \frac{n}{k-1}\right)^{\frac{2nk}{2n-1}} \left(\frac{2n+1}{2n-1}\right)^{nk}\\
\notag & \leq \exp \left( \left(\frac{ \ln \left( \frac{k}{2n-1} + 2 \right) }{nk} + \frac{ r \ln 2 + \ln n }{nk} + \frac{2}{2n-1} \ln (2n)  + \ln \left(\frac{2n+1}{2n-1} \right)  \right) nk \right) \\
\notag & \leq \exp \left( \left( \frac{\ln k}{k} + \frac{2}{2n-1} \ln 2   + \frac{2}{2n-1} \ln \left(2n \right) + \ln \left(\frac{2n+1}{2n-1} \right) \right) nk \right) \\
\notag & \leq \exp \left( \left( \frac{\ln k}{k} + \frac{4}{2n-1} \ln 2 + \frac{2}{2n-1}\ln n + \ln \left(\frac{2n+1}{2n-1} \right)  \right) nk \right) \\
\notag & \leq \exp \left( 10 k \ln n\right),
\end{split}
\end{equation}

\noindent where the last inequality holds for large $k$, in particular $k > n$, and since $\ln  ( 1 + x ) \leq x$ for $x > 0$. Let $t = 2 A \sqrt{10} \sqrt{\ln n / n }$. Then, 
\begin{equation}
 \begin{split}
\notag  |\mathcal{B}_{k, n} |  \exp \left( - \frac{t^2 k n } {A^2} \right) & \leq \exp \left( 10 k \ln n\right) \exp \left( - \frac{t^2 k n } {A^2} \right) \\
\notag & \leq \exp (- 30 k \ln n).
 \end{split}
\end{equation}

\noindent   Next, note that, as $k \to \infty$,  $\mathbb{E} \left[ LC_{kn} / (kn) \right] \to \gamma^*$  and that 
\begin{equation}
\notag \frac{1}{k} \left\lceil \frac{2kn}{ 2n-1 }\right\rceil  \leq \frac{1}{k} \left( \frac{2kn}{2n-1} + 1 \right) \to \frac{2n}{2n-1}. 
\end{equation}

\noindent  Recall also that $\beta^* = \lim_{n \to \infty} \beta(n) = \lim_{k \to \infty} \beta(kn)$. Then~\eqref{eq:hoeff} implies:
\begin{equation}\label{eq:hoeff_mod}
\frac{2n}{2n-1} \left( \frac{\mathbb{E}[LC_{2n}] }{2n} + 2 \beta^*\right) \geq \gamma^* - 2 A \sqrt{10} \sqrt{\frac{\ln n }{ n} },
\end{equation}

\noindent  and finally:
\begin{align}
\notag \frac{\mathbb{E}[LC_{2n}] }{2n} & \geq \frac{2n-1}{2n}\left( \gamma^* - 2 A \sqrt{10} \sqrt{\frac{\ln n}{ n} }\right) - 2 \beta^*\\
& \geq \gamma^* - 2 \beta^* - 2 A \sqrt{10} \sqrt{\frac{\ln n}{n}}  - \frac{1}{2n}.
\end{align}

\noindent To get the result for words of odd length note that by~\eqref{eq:hoeff_mod},
\begin{align}
\notag \frac{\mathbb{E}[LC_{2n+1}]}{2n+1} & \geq \frac{\mathbb{E}[LC_{2n}]}{2n+1}\\
\notag & \geq \frac{2n-1}{2n+1} \left( \gamma^* - 2 A \sqrt{10} \sqrt{\frac{\ln n}{ n} }\right)  - \frac{2n}{2n+1} 2 \beta^*  \\
\notag & \geq \gamma^* -2 \beta^* -  2 A \sqrt{10} \sqrt{\frac{\ln n}{ n} } - \frac{2}{2n+1}.
\end{align}

\noindent Of course, these last bounds are only of interest, for $n$ large enough, if $\gamma^* > 2 \beta^*$. Otherwise, we get the trivial lower bound $0$ (see  Remark~\ref{rem:final} below).  One is then left with slightly modifying the constants to get~\eqref{rate_bound_l}. The extra term on the right hand side in~\eqref{rate_bound_l} accounts for the difference in initial distributions ~\eqref{eq:mean_coupling}.

\noindent The proof of the upper bound~\eqref{rate_bound_u}, where symmetry is not needed, follows by combining Fekete's lemma (see~\cite{St}) with~\eqref{eq:mean_coupling} and~\eqref{eq:doeb}.

\end{proof}

\begin{rem}\label{rem:final} (i) Recall that the $\beta-$mixing coefficient $\beta(n)$ is a measure on the dependency between $(X_1, \ldots, X_n)$ and $(Y_1, \ldots, Y_n)$. The bounds in Theorem~\ref{rate_main} rely on $\beta^* \coloneqq \lim_{n \to \infty} \beta(n)$ which somehow quantifies a weak dependency requirement and  $\beta^* \neq 0$ unless the sequences $X$ and $Y$ are independent. Note also that the lower bound in Theorem~\ref{rate_main} is meaningful only if $2\beta^* < \gamma^*$. Besides the independent case, there are instances for which this condition is satisfied. For example, let  $X$ and $Y$ be both Markov chains with $L$ states and with the same transition matrix $P$, where some rows of $P$ are equal to $(1, 1, 1, \ldots, 1) / L$, i.e., such that  there exists  a set of states $\mathcal{L}$ such that the transition probability  between each one of these states is uniform. Let the initial distribution of $X_1$ be $\mu$ with $\mu(x) = 0$ if $x \notin \mathcal{L}$ and assume that $Y_1 = X_1$. Then the sequence $\tilde{Y}$ defined, for all $n$, via $\tilde{Y}_i = Y_i$, for $i \geq 1$ while $Y_1$ is distributed according to $\mu$ will be such that $\tilde{Y}^{(n)}$ and $Y^{(n)}$ have the same distribution. Moreover for all $n$, $\tilde{Y}^{(n)}$ and $X^{(n)}$ will be independent and $\mathbb{P}(\tilde{Y}^{(n)} \neq Y^{(n)} ) \geq \beta(n)$, but $\mathbb{P}(\tilde{Y}^{(n)} \neq Y^{(n)} ) = \mathbb{P}(Y_1 \neq \tilde{Y}_1)$ which can be made as small as desired for a suitable choice of $\mu$. Thus the lower bound in Theorem~\ref{rate_main} holds and is meaningful.

(ii) There are instances when the lower bound in Theorem~\ref{rate_main} is vacuous. Such a case is when $X_i = Y_i$ for all $i \geq 1$ and the $X_i$ are independent and uniformly distributed  over the letters in $\mathcal{A}$. Then, it is clear that $\gamma^* = 1$ whereas one  shows that
\begin{align}
\notag \beta(n) = 1 - \frac{1}{|\mathcal{A}|^n},
\end{align} 
\noindent and so $\beta^* = 1$. In this case the lower bound in~\eqref{rate_bound_l} is a negative quantity.

(iii) Theorem~\ref{rate_main} continues to hold for Markov chains with a general state space $\Lambda$. Indeed, the Hoeffding inequality~\eqref{hoeff} is true when $\Lambda$ is a Polish space. The exponential decay~\eqref{eq:doeb} holds when $\Lambda$ is \emph{petite}, i.e., when there exist a positive integer $n_0$, $\ep > 0$ and a probability measure $\nu$ on $\Lambda$ such that $P^{n_0} (x, A) \geq \ep \nu(A)$, for every measurable $A$ and $x \in \Lambda$, and where $P^{n_0}(x, A)$ is the $n_0-$step transition law of the Markov chain (see~\cite[Theorem 8]{Ros}).

\end{rem}

\noindent  When $X$ and $Y$ are generated by independent hidden Markov models. Then the following variant of Theorem~\ref{rate_main} holds (for a sketch of  proof, see the Appendix).

\begin{cor}\label{cor_indep}  Let $(Z_X, X)$ and $(Z_Y, Y)$ be two independent  hidden Markov models, where the latent chains $Z_X$ and $Z_Y$ have the same initial distribution, transition matrix and emission probabilities.  Then, for all $n \geq 2$,
\begin{equation}\label{rate_bound_l_ind}
 \frac{\mathbb{E}[LC_n]}{n} \geq \gamma^*  -  C \sqrt{\frac{\ln n}{n} } - \frac{2}{n}  - (1 - \mathbf{1}_{\mu = \pi}) \left( \frac{1}{\sqrt{n}}+  c\alpha^{\sqrt{n}}   \right) , 
\end{equation}

\noindent where $\alpha \in (0,1), c > 0$ are constants as in~\eqref{eq:doeb} and $C > 0$. All constants  depend on the parameters of the model but not on $n$. Moreover with the same $\alpha$ and $c$, 
\begin{align}
 \label{rate_bound_u_ind}\frac{\mathbb{E}[LC_n]}{n} \leq \gamma^*  + (1 - \mathbf{1}_{\mu = \pi}) \left( \frac{1}{\sqrt{n}}+ c \alpha^{\sqrt{n}}  \right).
\end{align}
\end{cor}

\noindent As mentioned in the end of the proof of Theorem~\ref{rate_main}, the symmetry of the distribution of $(X_i, Y_i)$ is used only for proving the lower bound. Let 
\begin{align}
\notag h(n) \coloneqq   \max_{m \in [-n, n]}  \left( 2 \sum_{i = 1}^{n-m} \mathbb{P}(X_i \neq Y_i) + \sum_{i = n-m+1}^{n+m} \mathbb{P}(X_i \neq Y_i) \right).
\end{align}

\noindent Then the following holds:

\begin{prop}\label{prop_asymetry}  Let $(Z, (X, Y))$ be a hidden Markov model, where the sequence $Z$ is an aperiodic time homogeneous and irreducible Markov chain with  finite state space $\mathcal{S}$. Then, for all $n \geq 2$,
\begin{equation}\label{rate_bound_l_sym}
 \frac{\mathbb{E}[LC_n]}{n} \geq \gamma^* - \frac{h(n)}{n}  - 2 \beta^* -  C \sqrt{\frac{\ln n}{n} } - \frac{2}{n}  - (1 - \mathbf{1}_{\mu = \pi}) \left( \frac{1}{\sqrt{n}}+  c\alpha^{\sqrt{n}}   \right).
\end{equation}

\end{prop} 

\noindent For a sketch of  proof  of this proposition, and some comments on $h(n)$, we again refer the reader to the Appendix.



\appendix

\section{}

First, as asserted in Remark~\ref{Hoeff_HMM} (ii), we provide two proofs of the fact that the mixing time $\tau(\ep)$ of $(Z_n, X_n, Y_n)_{n \geq 1}$ is the same as  the mixing time $\tilde{\tau}(\ep)$ of the chain $(Z_n)_{n \geq 1}$.

\begin{proof}[Proof 1]  let $\tilde{T} = (\tilde{T}_n)_{n \geq 1}$ be a Markov chain with  finite state space $\mathcal{S}$. Each $\tilde{T}_i$ emits an observed variable $T_i$ according to some probability distribution that depends only on the state $\tilde{T}_i$. Let $T = (T_n)_{n \geq 1}$ and assume $T_i \in \mathcal{A}$ - a finite alphabet. Note that $(\tilde{T}, T)$ is a Markov chain; let $\tau(\ep)$ be its mixing time, and let $\tilde{\tau}(\ep)$ be the mixing time for the hidden chain $\tilde{T}$.  Then,
\begin{align}
\notag d_{TV} & ( \mathcal{L} ((\tilde{T}_{i+t}, T_{i+t}) | (\tilde{T}_i, T_i) = (x, u) ) , \mathcal{L} ((\tilde{T}_{i+t}, T_{i+t}) | (\tilde{T}_i, T_i) = (y, v) ))\\
\notag = & \quad \frac{1}{2} \sum_{(z, w)\in \mathcal{S} \times \mathcal{A}} \bigg| \mathbb{P} ((\tilde{T}_{i+t}, T_{i+t}) = (z, w) | (\tilde{T}_i, T_i) = (x, u) ) -\\
\notag & \quad \quad \quad \quad \quad \quad \quad \quad\quad \quad -  \mathbb{P} ((\tilde{T}_{i+t}, T_{i+t}) = (z, w) | (\tilde{T}_i, T_i) = (y, v) \bigg| \\
\notag = & \quad \frac{1}{2}  \sum_{(z, w) } \bigg| \mathbb{P}(\tilde{T}_{i+t}  = z |\tilde{T}_i = x) \mathbb{P} (z \to w) -  \mathbb{P}(\tilde{T}_{i+t}  = z |\tilde{T}_i = y) \mathbb{P} (z \to w) \bigg| \\
\notag = & \quad \frac{1}{2} \sum_{(z, w)} \mathbb{P}(z \to w) \bigg| \mathbb{P}(\tilde{T}_{i+t}  = z |\tilde{T}_i = x) -  \mathbb{P}(\tilde{T}_{i+t}  = z |\tilde{T}_i = y)\bigg| \\
\notag = &\quad  \frac{1}{2} \sum_{z \in \mathcal{S}} \bigg|  \mathbb{P}(\tilde{T}_{i+t}  = z |\tilde{T}_i = x) - \mathbb{P}(\tilde{T}_{i+t}  = z|\tilde{T}_i = y) \bigg| \\
\notag = &\quad  d_{TV} ( \mathcal{L} (\tilde{T}_{i+t} | \tilde{T}_i = x ) , \mathcal{L} (\tilde{T}_{i+t} | \tilde{T}_i = y )), 
\end{align}

\noindent where $\mathbb{P}(z \to w) \coloneqq \mathbb{P} (T_i = w | \tilde{T}_i = z)$, i.e., the probability that a state with value $z \in \mathcal{S}$ emits $w \in \mathcal{A}$.  By definition of $T$ and $\tilde{T}$ this last probability does not depend on $i$. Then $\sum_{w \in \mathcal{A}} \mathbb{P}(z \to w) = 1$. Therefore, $\overline{d}_{(\tilde{T}, T)}(t) = \overline{d}_{\tilde{T}} (t)$ and $\tau(\ep) =  \tilde{\tau}(\ep)$. 

\end{proof}

\begin{proof}[Proof 2]  An alternative approach to proving the result of Remark~\ref{Hoeff_HMM} (ii)  relies on coupling arguments and was kindly suggested by D. Paulin in personal communications with the authors. First, recall the following classical result~\cite[Proposition 4.7]{LPW}:

\begin{lem}\label{lem:coup} Let $\mu$ and $\nu$ be two probability distributions on $\Omega$. Then,
\begin{align}
\notag d_{TV}(\mu, \nu) = \inf \{ \mathbb{P}(X \neq Y) : (X, Y) \text{ is a coupling of } \mu \text{ and } \nu \}.
\end{align}

\noindent Moreover, there is a coupling $(X,Y)$ which attains the infimum and such a coupling is called optimal.

\end{lem}

 \noindent Let $(\tilde{T}^1, \tilde{T}^2)$ be an optimal coupling according to $d_{TV} ( \mathcal{L} (\tilde{T}_{t} | \tilde{T}_1 = x ) , \mathcal{L} (\tilde{T}_{t} | \tilde{T}_1 = y ))$ for some $x, y \in \mathcal{S}$, i.e., $\tilde{T}^1$ and $\tilde{T}^2$ are Markov chains with the same transition probability as $\tilde{T}$, $\tilde{T}_0^1 = x$, $\tilde{T}_0^2 = y$, and 
\begin{align}
\label{3.4_lower} \mathbb{P}  (\tilde{T}_t^1  \neq  \tilde{T}_t^2) = &  \quad d_{TV} ( \mathcal{L} (\tilde{T}_{t} | \tilde{T}_1 = x ) , \mathcal{L} (\tilde{T}_{t} | \tilde{T}_1 = y ))
\end{align}

\noindent Next let $T_t^1$ and $T_t^2$ be respectively distributed according to the distributions associated with $\tilde{T}_t^1$ and $\tilde{T}_t^2$ and be independent of all the other random variables. In addition, if for some $t \geq 1$,  $\tilde{T}_t^1 = \tilde{T}_t^2$, then $T_t^1 = T_t^2$. Then
\begin{align}
\notag \mathbb{P} (\tilde{T}_t^1 \neq \tilde{T}_t^2) = \mathbb{P} \left((\tilde{T}_t^1, T_t^1) \neq (\tilde{T}_t^2, T_t^2)\right),
\end{align}

\noindent and by Lemma~\ref{lem:coup}, for any $u, v \in \mathcal{A}$ and any $i \geq 1$, 
\begin{align}
\notag \mathbb{P} &  \left((\tilde{T}_t^1, T_t^1)  \neq (\tilde{T}_t^2, T_t^2)\right)  \\ 
\notag \geq & \quad   d_{TV} \left( \mathcal{L} ((\tilde{T}_{i+t}, T_{i+t}) | (\tilde{T}_i, T_i) = (x, u) ) , \mathcal{L} ((\tilde{T}_{i+t}, T_{i+t}) | (\tilde{T}_i, T_i) = (y, v) )\right).
\end{align}

\noindent Together with~\eqref{3.4_lower}, the above yields
\begin{align}
\notag d_{TV} ( \mathcal{L} (\tilde{T}_{t} | \tilde{T}_1 = x ) &, \mathcal{L} (\tilde{T}_{t} | \tilde{T}_1 = y ))  \geq \\ 
\notag \geq \quad d_{TV} \bigg( \mathcal{L} ((\tilde{T}_{i+t}, T_{i+t}) & | (\tilde{T}_i, T_i) = (x, u) )  , \mathcal{L} ((\tilde{T}_{i+t}, T_{i+t}) | (\tilde{T}_i, T_i) = (y, v) )\bigg)
\end{align}

\noindent Taking the $\sup$ over $x, y, u, v$ gives  $\overline{d}_{(\tilde{T}, T)}(t)  \leq \overline{d}_{\tilde{T}} (t)$.

\noindent For the reverse inequality, consider the optimal coupling $\left((\tilde{T}^1, T^1), (\tilde{T}^2, T^2)\right)$ according to $d_{TV} \left( \mathcal{L} ((\tilde{T}_{t}, T_{t}) | (\tilde{T}_1, T_1) = (x, u) ) , \mathcal{L} ((\tilde{T}_{t}, T_{t}) | (\tilde{T}_1, T_1) = (y, v) )\right)$, for some $x, y \in \mathcal{S}$ and $u,v \in \mathcal{A}$. Then,
\begin{align}
\label{3.4_upper1} \mathbb{P}  \left((\tilde{T}_t^1, T_t^1)  \neq (\tilde{T}_t^2, T_t^2)\right) =  d_{TV} \left( \mathcal{L} ((\tilde{T}_{t}, T_{t}) | (\tilde{T}_1, T_1) = (x, u) ) , \mathcal{L} ((\tilde{T}_{t}, T_{t}) | (\tilde{T}_1, T_1) = (y, v) )\right),  
\end{align}

\noindent and 
\begin{align}
\notag \mathbb{P}  \left((\tilde{T}_t^1, T_t^1)  \neq (\tilde{T}_t^2, T_t^2)\right)  \geq \mathbb{P}  (\tilde{T}_t^1 \neq \tilde{T}_t^2). 
\end{align}

\noindent However, by the  Lemma~\ref{lem:coup}, for any $i \geq 1$, 
\begin{align}
\label{3.4_upper2} \mathbb{P}  (\tilde{T}_t^1 \neq \tilde{T}_t^2) \geq d_{TV} ( \mathcal{L} (\tilde{T}_{i+t} | \tilde{T}_i = x ) , \mathcal{L} (\tilde{T}_{i+t} | \tilde{T}_i = y )).
\end{align}

\noindent Taking the $\sup$ in~\eqref{3.4_upper1} and~\eqref{3.4_upper2} gives,  $\overline{d}_{(\tilde{T}, T)}(t) \geq \overline{d}_{\tilde{T}} (t)$, and then $\overline{d}_{(\tilde{T}, T)}(t)  = \overline{d}_{\tilde{T}} (t)$.

\end{proof}

\begin{proof}[Proof of Corollary~\ref{cor_indep}]  The Hoeffding inequality~\eqref{hoeff} holds as long as $(Z, X, Y)$ is a Markov chain. In addition, $(X, Y)$ has to be symmetric (see proof of Proposition~\ref{prop_asymetry}) in order for~\eqref{eq:exchange} to hold. Again one such setting is when $X$ and $Y$ are two independent HMM  with the same transition matrix for the latent chain and same emission probabilities.  A rate of convergence result then  follows from arguments as in Section~\ref{s:RC}. The bound on $\mathcal{B}_{k,n}$ is the same, and there is  a Hoeffding type inequality for this model as per Remark~\ref{Hoeff_HMM} (i). One thing that differs is the bound~\eqref{eq:exchange}. In the present case it is much easier to get. When started at the stationary distribution, by  stationarity, independence and symmetry, one has:
\begin{equation}
\begin{split}
\notag LCS (X_{\nu_i} , \ldots, X_{\nu_{i+1} - 1}; Y_{\tau_i} , \ldots, Y_{\tau_{i+1} - 1}) & \leq LCS(X_1, \ldots , X_{n-m}, Y_1 , \ldots, Y_{n+m}) \\
\notag & = LCS(X_1, \ldots , X_{n+m}; Y_1 , \ldots, Y_{n-m}) \\
\notag & \leq \frac{1}{2} LC_{2n}.
\end{split}
\end{equation} 

\noindent In particular, there is no need to introduce mixing coefficients in this case ($\beta = 0$).  When the hidden chains are not started at the stationary distribution one gets an error as in~\eqref{eq:error_2HMM}. Then Theorem~\ref{rate_main} holds but with constants depending on the new model.  Moreover, this setting reduces to the one where $X$ and $Y$ are independent Markov chains by letting each state of the hidden chains emit a unique letter, which can further recover the iid case originally obtained in~\cite{A}.

\end{proof}

\begin{proof} [Proof of Proposition~\ref{prop_asymetry}] The symmetry of the distribution of $(X, Y)$ is only used  to get~\eqref{eq:4}, which entails that for any  $m \in \{-n+1, \ldots,  n-1\}$,  $LCS(X_1, \ldots, X_{n-m} ; $ $  Y_1, \ldots, Y_{n+m} )$ and $LCS(X_1, \ldots, X_{n+m} ; $ $ Y_1, \ldots, Y_{n-m} )$ are identically distributed and upper bounded by half of $LC_{2n}$. Such a result yields a comparison between $\mathbb{E}[LC_{2n}]$ and $\mathbb{E}[LC_{kn}]$, leading as $k \to \infty$, to a lower bound on $\mathbb{E}[LC_{2n}]$ involving $\gamma^*$.   Without assuming symmetry, the step~\eqref{eq:4} in obtaining~\eqref{eq:exchange} needs to be modified. One way to do so is to make  use of  the Lipschitz property of the LCS to get the following estimate:
\begin{align}
\notag LCS & (X_1, \ldots, X_{n-m} ; Y_1, \ldots, Y_{n+m}) \\
\notag = & \quad LCS  (X_1, \ldots, X_{n+m} ; Y_1, \ldots, Y_{n-m}) + \bigg( LCS  (X_1, \ldots, X_{n-m} ; Y_1, \ldots, Y_{n+m}) \\
\notag & \quad \quad \quad  - LCS  (X_1, \ldots, X_{n+m} ; Y_1, \ldots, Y_{n-m}) \bigg) \\
\notag \leq & \quad LCS  (Y_1, \ldots, Y_{n-m} ; X_1, \ldots, X_{n+m}) +  2 \sum_{i = 1}^{n-m} \mathbf{1}_{X_i \neq Y_i} + \sum_{i = n-m+1}^{n+m} \mathbf{1}_{X_i \neq Y_i}. 
\end{align}

\noindent Taking expectations, then~\eqref{eq:exchange} becomes

\begin{equation}
\notag \mathbb{E}  [ LCS (X_{\nu_i} , \ldots, X_{\nu_{i+1} - 1}; Y_{\tau_i} , \ldots, Y_{\tau_{i+1} - 1})] \leq \frac{1}{2} \bigg( \mathbb{E}[LC_{2n}] + h(n) \bigg)  + 2 n\beta(kn) , 
\end{equation}

\noindent where $h(n) \coloneqq   \max_{m \in [-n, n]} \left( 2 \sum_{i = 1}^{n-m} \mathbb{P}(X_i \neq Y_i) + \sum_{i = n-m+1}^{n+m} \mathbb{P}(X_i \neq Y_i) \right)$. This leads to a non-symmetric version of~\eqref{rate_bound_l}, namely,

\begin{equation}
\label{rate_bound_l_sym1}  \frac{\mathbb{E}[LC_n]}{n} \geq \gamma^* -  C \sqrt{\frac{\ln n}{n} } - \frac{ h(n) }{n} - 2 \beta^* - \frac{1}{n-2}  - (1 - \mathbf{1}_{\mu = \pi}) \left( \frac{1}{\sqrt{n}}+  c\alpha^{\sqrt{n}}   \right).
\end{equation}

\end{proof}

If  $h(n) = O(\sqrt{n  \ln n } )$, then the rate in~\eqref{rate_bound_l_sym} or~\eqref{rate_bound_l_sym1} will be the same as in~\eqref{rate_bound_l}. Such will be the case when $(Z', X)$ and $(Z'', Y)$ are two independent hidden Markov models and $Z = (Z', Z'') $ is a coupling of the two latent chains such that if  $Z'_i = Z_i''$, then $Z_j' = Z_j''$ for any $j  > i$. Then, $(Z, (X, Y))$ is a hidden Markov model where $X_i = Y_i$ once the two latent chains have met, and by~\eqref{eq:doeb} $h(n) = O(\sqrt{n \log n})$. 

However, $h(n)$ can be much larger, e.g, of order $n$. A case in hand is when the $X_i$ and $Y_i$ are  iid Bernoulli random variables with parameters $1/3$ and $1/2$ respectively. Then $\mathbb{P}(X_i \neq Y_i) = \mathbb{P}(X_i = 0, Y_i = 1) + \mathbb{P} (X_i = 1, Y_i =0) = 1/6 + 2/6 = 1/2$, for all $i \geq 1$,  and 
\begin{align} 
\notag  \left( 2 \sum_{i = 1}^{n-m} \mathbb{P}(X_i \neq Y_i) + \sum_{i = n-m+1}^{n+m} \mathbb{P}(X_i \neq Y_i) \right) = \big( 2 (n-m) 1/2 + (2m) 1/2\big) = n.
\end{align}

\noindent Note also that when $X = (X_i)_{i \geq 1}$ and $Y = (Y_i)_{i \geq 1}$ are independent sequences of random variables, the symmetry assumption is equivalent to $X$ and $Y$ being identically distributed.

\acks

This material is based in part upon work supported by the National Science Foundation under Grant No. 1440140, while the first author was in residence at the Mathematical Sciences Research Institute in Berkeley, California, during the Fall semester of 2017. His research was also supported in part by the grants \#246283 and \# 524678 from the Simons Foundation. The second author was  partially supported from the TRIAD NSF grant (award 1740776). Both authors are grateful to J. Spouge and S. Eddy for encouragements and correspondence on the relevance of the hidden Markov models in computational biology.


\begin{thebibliography}{9}

\bibitem{A} \textsc{Alexander, K.} (1994). The rate of convergence of the mean length of the longest common subsequence. \textit{Ann. Appl. Probab.}  \textbf{4} 1074--1082. 

\bibitem{B} \textsc{Berbee, H.C.P.} (1979). Random walks with stationary increments and renewal theory. Mathematical Centre Tracts, 112. \textit{Mathematisch Centrum, Amsterdam}. 

\bibitem{B_a} \textsc{Bradley, R} (1983) Approximation theorems for strongly mixing random variables. Michigan Math J. \textbf{30} 69--81.

\bibitem{B_s} \textsc{Bradley, R.} (2005) Basic properties of strong mixing conditions. A survey and some open questions. \textit{Prob. Surveys} \textbf{2} 104--144.

\bibitem{Br} \textsc{Bradley, R.} (2007). Introduction to strong mixing conditions. Vol. 1. \textit{Kendrick Press, Heber City, UT}. 

\bibitem{CS} \textsc{Chv\'atal, V., Sankoff, D.} (1975). Longest common subsequences of two random sequences.  \textit{J. Appl. Probability} \textbf{12} 306--315. 

\bibitem{D} \textsc{Doeblin, W.} (1938). Expos\'e de la th\'eorie des cha\^ines simples constantes de Markoff \`a un nombre fini d'\'etats, \textit{Revue Math. de l'Union Interbalkanique} \textbf{2}, 77--105.


\bibitem{PD} \textsc{ Doukhan, P.} (1994). Mixing. Properties and examples. Lecture notes in Statistics, 85. \textit{Springer-Verlag, New York.}


\bibitem{DEKM} \textsc{Durbin, R., Eddy, S., Krogh, A., Mitchison, G.} (1998). Biological sequence analysis, \textit{Cambridge university press, Cambridge}.

\bibitem{F} \textsc{Feller, W.} (1968). An introduction to probability theory and its applications. Vol. 1. Third edition. \textit{John Wiley \& Sons, New York-London-Sydney}. 

\bibitem{G} \textsc{Goldstein, S.} (1979). Maximal coupling. \textit{Z. Wahrsch. verw. Gebiete} \textbf{46} 193--204. 

\bibitem{LMT} \textsc{Lember, J., Matzinger, H., Torres, F.} (2012). The rate of convergence of the mean score in random sequence comparison. \textit{Ann. Appl. Probab.} \textbf{22}, no. 3, 1046--1058. 

\bibitem{LPW} \textsc{Levin, D., Peres, Y., Wilmer, E.} (2008). Markov chains and mixing times. \textit{AMS, Providence, RI}




\bibitem{DP} \textsc{Paulin, D.} (2015). Concentration inequalities for Markov chains by Marton couplings and spectral methods. \textit{Electron. J. Probab.} \textbf{20}, no. 79. 

\bibitem{R} \textsc{Rhee, W.} (1995). On rates of convergence for common subsequences and first passage time. \textit{Ann. Appl. Probab.} \textbf{5}, no. 1, 44--48. 

\bibitem{ER} \textsc{Rio, E.} (2017) Asymptotic theory of weakly dependent random processes. Probability theory and stochastic modelling, 80. \textit{Springer-Verlag, GmbH Germany}

\bibitem{Ros} \textsc{Roberts, G., Rosenthal, J.} (2004) General state space Markov chains and MCMC algorithms. \textsc{Prob. Surveys} \textbf{1}, 20--71. 

\bibitem{S} \textsc{Steele, M.} (1982). Long common subsequences and the proximity of two random strings. \textit{SIAM J. Appl. Math.} \textbf{42}, no. 4, 731--737. 

\bibitem{St} \textsc{Steele, M.}(1997). Probability theory and combinatorial optimization. \textit{SIAM, Philadelphia, PA}, 18--21.

\bibitem{Thorisson} \textsc{Thorisson, H.} (2000). Coupling, stationarity and regeneration. Probability and its applications (New York). \textit{Springer-Verlag, New York.} 

\end{thebibliography}
\end{document}